\documentclass[a4paper,reqno,10pt,oneside]{amsart}

\usepackage{amsfonts,amssymb,latexsym,xspace,epsfig,graphics,color}
\usepackage{amsmath,enumerate,stmaryrd,xy}
\usepackage{bbm}
\usepackage{natbib}
\numberwithin{figure}{section}
\usepackage{subfigure}
\usepackage{dsfont}
\usepackage{a4wide}
\usepackage{enumitem}   

\usepackage{tikz}
\usepackage{tkz-berge}
\usetikzlibrary{arrows,snakes,backgrounds,trees}
\usetikzlibrary{matrix,decorations.pathreplacing,positioning}
\usepackage{stmaryrd}
\newtheorem{theorem}{Theorem}[section]

\newtheorem{prop}[theorem]{Proposition}                          
\newtheorem{lemma}[theorem]{Lemma}
\newtheorem{corollary}[theorem]{Corollary}
\newtheorem{defn}[theorem]{Definition}

\newtheorem{exa}{Example}[section]

\def\T+{{\mathbb T_d^+}}

\def\A{\mathcal{A}}

\def\rank {\mathop {\rm rank}\nolimits}

\def\Der {\mathop {\rm Der}\nolimits}
\def\T {\mathcal{T}}

\bibdata{prob}
\bibstyle{alpha} 

\title[]{Characterization theorems for the spaces of derivations of evolution algebras associated to graphs}

\author{Paula Cadavid, Mary Luz Rodi\~no Montoya and Pablo M. Rodriguez}

\date{}

\address{
\newline
Paula Cadavid, Pablo M. Rodriguez
\newline
Instituto de Ci\^encias Matem\'aticas e de Computa\c{c}\~ao, Universidade de S\~ao Paulo
\newline  
Caixa Postal 668, 13560-970 S\~ao Carlos, SP, Brazil
\newline
e-mails: pacadavid@usp.br,  pablor@icmc.usp.br
\newline
\newline
Mary Luz Rodi\~no Montoya
\newline
Instituto de Matem\'aticas - Universidad de Antioquia 
\newline
Calle 67 N$^{\circ}$ 53-108, Medell\'in, Colombia
\newline 
e-mail: mary.rodino@udea.edu.co 
}


\begin{document}
\maketitle

\begin{abstract}
It is well-known that the space of derivations of $n$-dimensional evolution algebras with non-singular matrices is zero. On the other hand, the space of derivations of evolution algebras with matrices of rank $n-1$ has also been completely described in the literature. In this work we provide a complete description of the space of derivations of  evolution algebras associated to graphs, depending on the twin partition of the graph. For graphs without twin classes with at least three elements we prove that the space of derivations of the associated evolution algebra is zero. Moreover, we describe the spaces of derivations for evolution algebras associated to the remaining families of finite graphs. It is worth pointing out that our analysis includes examples of finite dimensional evolution algebras with matrices of any rank. 

\end{abstract}

\section{Introduction}

The Theory of Evolution Algebras is a current and active field of research with many applications and connections to other areas of mathematics. These algebras are non-associative algebras and form a special class of genetic algebras. We refer the reader to \cite{PMP2,camacho/gomez/omirov/turdibaev/2013,Elduque/Labra/2015,nunez/2014,COT} and references therein for an overview of recent results. An $n$-dimensional evolution algebra is defined as follows.

\smallskip
\begin{defn}\label{def:evolalg}
Let $\A:=(\A,\cdot\,)$ be an algebra over a field $\mathbb{K}$. We say that $\A$ is an evolution algebra if it admits a finite basis $S:=\{e_1,\ldots , e_n\}$, such that
\begin{equation}\label{eq:ea}
\begin{array}{cl}
e_i \cdot e_i =\sum_{k \in S} c_{ik} e_k,&\text{ for }i\in\{1,\ldots,n\},\\[.2cm]
e_i \cdot e_j =0,&\text{ for }i,j\in\{1,\ldots,n\}\text{ such that }i\neq j.
\end{array}
\end{equation} 
 \end{defn}
 
 \smallskip
 A basis $S$ satysfying \eqref{eq:ea} is called natural basis of $\mathcal{A}$. The scalars $c_{ik}\in \mathbb{K}$ are called the structure constants of $\mathcal{A}$ relative to $S$ and $(c_{ik})$ is called the matrix of $\A$ relative to $S$. {\color{black} In what follows, we always assume that $\text{char}(\mathbb{K})=0$.} The first and best general reference of a Theory of Evolution Algebras is the book of Tian \cite{tian}, were the author states the first properties for these structures as well as an interesting correspondence between this subject and Probability Theory; more precisely, with the Theory of Markov Chains. 
 
In this work we are interested in the problem of characterizing the space of derivations of a given evolution algebra. The advantage of describing this space lies in the fact that it induces a Lie algebra which may be used as a tool for studying the structure of the original algebra, see \cite{holgate} and the references given there for a deeper discussion of this subject in genetic algebras. A derivation is defined, as usual, as follows. 

\smallskip
\begin{defn}
Let $\A$ be an (evolution) algebra. A derivation of $\A$ is a linear operator $d: \A \rightarrow  \A$ such that 
$$d(u \cdot v)= d(u) \cdot  v + d(v)\cdot u,$$
for all $u,v \in \A$. The space of all derivations of the (evolution) algebra $\A$ is denoted by $\Der(\A)$.
\end{defn}

\smallskip

The space of all the derivations of an algebra has already been completely described for some genetic algebras, see \cite{COT,costa1,costa2,gonshor,gonzales,peresi}. {\color{black}We refer the reader to \cite{holgate} for a reference providing the interpretation of a derivation on a genetic algebra by means of the equality of two genetically meaningful expressions.}

For the case of an evolution algebra, a complete characterization of such space is still an open question. In \cite{tian} it is observed that a linear operator $d$ such that $d(e_i)=\sum_{k=1}^{n} d_{ik}e_k$ is a derivation of the evolution algebra $\A$ if, and only if, it satisfies the following conditions: 

\begin{alignat}{1}
c_{jk}d_{ij}+ c_{ik}d_{ji}=0, &\,\,\, \text{ for }i,j,k\in\{1,\ldots,n\}\text{ such that }i\neq j,\label{eq:der1}\\
\sum_{k=1} ^{n} c_{ik}d_{kj}=2c_{ij}d_{ii}, &\,\,\, \text { for }i,j\in\{1,\ldots,n\}.\label{eq:der2}
\end{alignat}

Therefore, Equations \eqref{eq:der1} and \eqref{eq:der2} are the starting point whether one want to obtain a description of the space of derivations for a particular evolution algebra $\A$. As far as we known, the only almost-general result regarding the space of derivations of a finite dimensional evolution algebra is obtained by \cite{COT}. In that work the authors {\color{black}consider $\mathbb{K}=\mathbb{C}$ and they} prove that the space of derivations of $n$-dimensional evolution algebras with non-singular matrices is zero; in addition, they describe the space of derivations of evolution algebras with matrices of rank $n-1$. {\color{black}More recently, \cite{farrukh} describes the space of derivations of evolution algebras with maximal nilindex restricted to the case where $\text{char}(\mathbb{K})=0$.} 

 The aim of this paper is twofold. On the one hand, we give a complete characterization of the space of derivations in a special family of $n$-dimensional evolution algebras, i.e. the ones associated to finite graphs. Furthermore, our results cover examples of evolution algebras with matrices of rank $k$, where $k\in \{2,\ldots,n\}$. On the other hand, we continue in studying evolution algebras associated to graphs, a subject which was initiated by \cite{nunez/2013,tian}, and considered more recently by \cite{PMP,PMP2}. We shall see that the twin partition of the given graph is crucial for characterizing the space of derivations of its respective evolution algebra. We emphasize that this provides an alternative to the criteria developed by \cite{COT}, for the case of evolution algebras associated to graphs. 
 
The paper is organized as follows. In Section 2 we review some of the standard notation on Graph Theory,  and we state our main results. Section 3 is devoted to the proofs of these results.

\section{Main results}

\subsection{Evolution algebra of a graph}

In order to present our results we start with some standard definitions and notation of Graph Theory. A finite graph $G$ with $n$ vertices is a pair $(V,E)$ where $V:=\{1,\ldots,n\}$ is the set of vertices and $E:=\{(i,j)\in V\times V:i\leq j\}$ is the set of edges. If $(i,j)\in E$  or $(j,i)\in E$ we say that $i$ and $j$ are {\it neighbors}, and we denote the set of neighbors of a vertex $i$ by $\mathcal{N}(i)$. In general, given a subset $U\subset V$, we denote $\mathcal{N}(U):=\{j\in V: j\in \mathcal{N}(i) \text{ for some }i\in U\}$, and $U^c:=\{j\in V: j\notin U\}$. The {\it degree} of vertex $i$, denoted by $\deg(i)$, is  the cardinality of the set $\mathcal{N}(i)$. The {\it adjacency matrix}  $A_G=(a_{ij})$ of $G$ is an $n\times n$ symmetric matrix such that $a_{ij}=1$ if $i\in \mathcal{N}(j)$, and $a_{ij}=0$ otherwise. Note that, for any $k \in V$, $\mathcal{N}(k):=\{\ell \in V: a_{k\ell}=1\}$. We say that vertices $i$ and $j$ of a graph $G$ are {\it twins}\footnote{For the sake of simplicity we use the nomenclature of twins but, since we consider graphs without loops, it coincides with the concept of non-adjacency twins, or false twins, depending of the reference.} if they have exactly the same set of neighbors, i.e. $\mathcal{N}(i)=\mathcal{N}(j)$. We notice that by defining the relation $\sim_{t}$ on the set of vertices $V$ by $i\sim_{t} j$ whether $i$ and $j$ are twins, then $\sim_{t}$ is an equivalence relation. An equivalence class of the twin relation is referred to as a {\it twin class}. In other words, the twin class of a vertex $i$ is the set $\{j\in V:i \sim_{t} j\}$. The set of all twin classes of $G$ is denoted by $\Pi(G)$ and it is referred to as the {\it twin partition} of $G$. A graph is {\it twin-free} if it has no twins. All the graphs we consider are connected, i.e. for any $i,j\in V$ there exists a sequence of vertices $i_0,i_1,i_2,\ldots,i_n$ such that $i_0=i$, $i_n=j$ and $i_{k+1}\in \mathcal{N}(i_{k})$ for all $k\in\{0,1,\ldots,n-1\}$. For simplicity, we consider only graphs which are simple, i.e. without multiple edges or loops. 

\smallskip
The evolution algebra associated to a given graph $G$, and denoted by $\A(G)$, is defined by letting $c_{ij}= a_{ij}$ for any $i,j\in V$, see \cite[Section 6.1]{tian}. 

\smallskip
\begin{defn}\label{def:eagraph}
Let $G=(V,E)$ be a graph with adjacency matrix given by $A_G=(a_{ij})$. The evolution algebra associated to $G$ is the algebra $\A(G)$ with natural basis $S=\{e_i: i\in V\}$, and relations

\[
 \begin{array}{ll}\displaystyle
e_i \cdot e_i = \sum_{k\in V} a_{ik} e_k, \text{  for  }i \in  V,\\[.5cm]
e_i \cdot e_j =0,\text{ if }i\neq j.
\end{array}
\] 
\noindent

\end{defn}

\smallskip
We notice that another way of stating the relations for $e_i \cdot e_i$ in the previous definition is to say:
$$e_i \cdot e_i = \sum_{k\in \mathcal{N}(i)} e_k, \text{ for  }i \in  V.$$

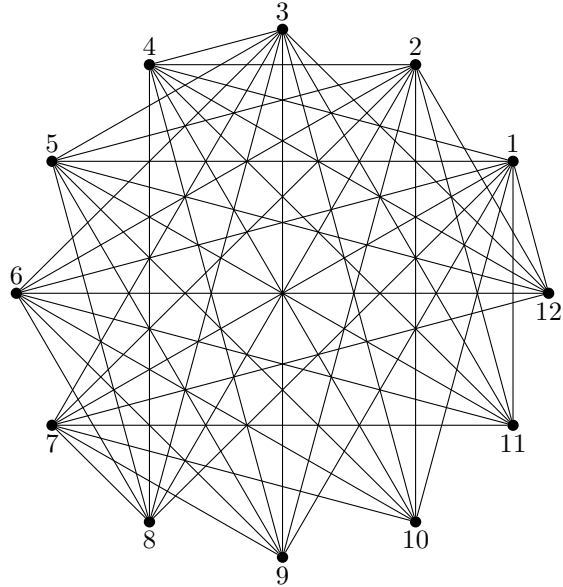
\begin{figure}[!h]
\begin{center}

\begin{tikzpicture}[scale=.7]
\GraphInit[vstyle=Simple]
    \SetGraphUnit{1}
    \tikzset{VertexStyle/.style = {shape = circle,fill = black,minimum size = 2.5pt,inner sep=1.5pt}}
  \begin{scope}[xshift=10cm]
\grEmptyCycle[prefix=a,RA=5]{12}%

\draw (a1)--(a4);
\draw (a1)--(a5);
\draw (a1)--(a6);
\draw (a1)--(a7);
\draw (a1)--(a8);
\draw (a1)--(a9);
\draw (a1)--(a10);
\draw (a1)--(a11);
\draw (a1)--(5,0);

\draw (a2)--(a4);
\draw (a2)--(a5);
\draw (a2)--(a6);
\draw (a2)--(a7);
\draw (a2)--(a8);
\draw (a2)--(a9);
\draw (a2)--(a10);
\draw (a2)--(a11);
\draw (a2)--(5,0);

\draw (a3)--(a4);
\draw (a3)--(a5);
\draw (a3)--(a6);
\draw (a3)--(a7);
\draw (a3)--(a8);
\draw (a3)--(a9);
\draw (a3)--(a10);
\draw (a3)--(a11);
\draw (a3)--(5,0);

\draw (a4)--(a8);
\draw (a4)--(a9);
\draw (a4)--(a10);
\draw (a4)--(a11);
\draw (a4)--(5,0);

\draw (a5)--(a8);
\draw (a5)--(a9);
\draw (a5)--(a10);
\draw (a5)--(a11);
\draw (a5)--(5,0);

\draw (a6)--(a8);
\draw (a6)--(a9);
\draw (a6)--(a10);
\draw (a6)--(a11);
\draw (a6)--(5,0);

\draw (a7)--(a8);
\draw (a7)--(a9);
\draw (a7)--(a10);
\draw (a7)--(a11);
\draw (a7)--(5,0);

\node[above] at (a1) {$1$};
\node[above] at (a2) {$2$};
\node[above] at (a3) {$3$};
\node[above] at (a4) {$4$};
\node[above] at (a5) {$5$};
\node[above] at (a6) {$6$};

\node[below] at (a7) {$7$};
\node[below] at (a8) {$8$};
\node[below] at (a9) {$9$};
\node[below] at (a10) {$10$};
\node[below] at (a11) {$11$};
\node[below] at (5,0) {$12$};

\end{scope}
\end{tikzpicture}
\caption{Complete $3$-partite graph $K_{3,4,5}$. Partitions are given by $V_1=\{1,2,3\}$, $V_2=\{4,5,6,7\}$ and $V_3=\{8,9,10,11,12\}$.}\label{fig:npartite}
\end{center}
\end{figure}

\smallskip
\begin{exa}\label{exa:npartite}
Let $K_{a_1, a_2, \ldots, a_m}$ be a complete $m$-partite graph, with $m\geq 2$, and partitions of sizes $a_1, a_2, \ldots, a_m$, where $a_i \geq 1$ for $i\in \{1,2,\ldots,m\}$. This is a graph for which the set of vertices is partitioned into $m$ disjoint sets, with sizes $a_1, a_2, \ldots, a_m$, in such a way that there is no edge connecting two vertices in the same subset, and every possible edge that could connect vertices in different subsets is part of the graph, see Fig. \ref{fig:npartite}. The resulting evolution algebra associated to this graph, denoted by $\mathcal{A}(K_{a_1, a_2, \ldots, a_m})$, is given by the natural basis 
$\{e_1,  \ldots, e_{a_1 + \cdots + a_m}  \}$ and relations:
\smallskip
\[
\begin{array}{cl}\displaystyle

 e_i^2   = \sum_{j\notin V_{\ell}} e_{j}, &\text{ for }i\in V_{\ell}, \text{ and }\ell\in\{1,2,\ldots,m\}, \\[.4cm]
  e_i \cdot e_j =  0, &\text{ for }i\neq j.
\end{array}
\]

\end{exa}

\smallskip
The first definitions of evolution algebras associated to some families of graphs have been formalized by \cite{nunez/2014,nunez/2013}, were the authors also study properties related to these algebras. More recently, \cite{PMP,PMP2} study the existence of isomorphisms between these algebras and the evolution algebras associated to the random walk on the same graph. This is a subject that has been remained as an open question since 2008, see \cite{tian,tian2} for further details. We refer the reader to  \cite{Elduque/Labra/2015} for another reference related to the interplay between evolution algebras and graphs. In that work, the authors consider a digraph associated to any evolution algebra, which leads to new algebraic results on this class of algebras and are useful to provide new natural proofs of some known results.

\bigskip
\subsection{Space of derivations of the evolution algebra of a graph}

Our aim is to characterize the space of derivations of an evolution algebra associated to a graph. In the rest of the paper we shall assume that $G=(V,E)$ is a finite graph with $n$ vertices and $d \in \Der(\A(G))$ is such that 
\begin{equation}\label{eq:derexp}
d(e_i)=\sum_{k=1}^{n}d_{ik}e_k,
\end{equation} 

\noindent
for any $i\in V.$ We shall consider graphs with at least three vertices since in other case $d=0$ by a direct application of \eqref{eq:der1} and \eqref{eq:der2}. Our description of $\Der(\A(G))$ will depends on the twin partition of $G$. More precisely,  our characterization will depends on the set  $\Gamma_3(G)\subset \Pi(G)$ formed by all twin classes of $G$ with at least three vertices. We start with a necessary condition for $d\neq 0$.

\begin{theorem}\label{prop:necessary}
Let $G=(V,E)$ be a finite graph, with $|V|\geq 3$, and let $d\in \Der(\A(G))$. If $\Gamma_3(G)=\emptyset$ then $d=0$.
\end{theorem}

The previous result establishes a criterium to find evolution algebras associated to graphs for which the only derivation is the null map, see Fig. \ref{FIG:graphsderzero} for an illustration of graphs with $\Gamma_3(G)=\emptyset$. It is worth pointing out that, in the case of evolution algebras associated to graphs, it provides an alternative to the criteria obtained by \cite{COT} where the rank of the matrix of $\A(G)$ relative to $S$ should be computed. As a sideline, as we shall see later, our criteria depends on the notion of twin vertices, which has been used in the literature for understanding combinatorial properties of graphs, see for instance \cite{abadi, charon, levit} and references therein. Thus, our work also provides a new application of the twin partition of a graph.\\

\bigskip
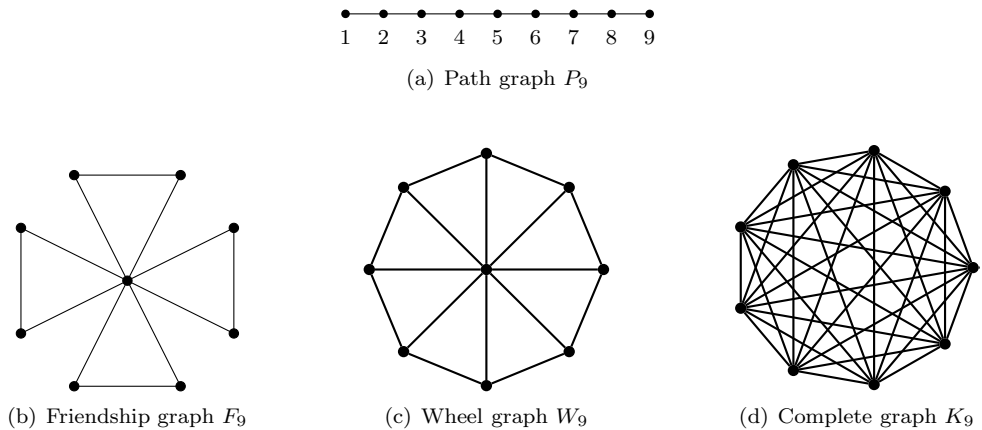
\begin{figure}[h]
\begin{center}

\subfigure[][Path graph $P_9$]{
\begin{tikzpicture}[scale=0.5]

\draw (-3,-2) -- (5,-2);
\draw (-3,-2.2) node[below,font=\footnotesize] {$1$};


\draw [very thick] (-3,-2) circle (2pt);
\filldraw [black] (-3,-2) circle (2pt);
\draw (-2,-2.2) node[below,font=\footnotesize] {$2$};
\draw [very thick] (-2,-2) circle (2pt);
\filldraw [black] (-2,-2) circle (2pt);
\draw [very thick] (-1,-2) circle (2pt);
\filldraw [black] (-1,-2) circle (2pt);
\draw (-1,-2.2) node[below,font=\footnotesize] {$3$};
\draw [very thick] (0,-2) circle (2pt);
\filldraw [black] (0,-2) circle (2pt);
\draw (0,-2.2) node[below,font=\footnotesize] {$4$};
\draw [very thick] (1,-2) circle (2pt);
\filldraw [black] (1,-2) circle (2pt);
\draw (1,-2.2) node[below,font=\footnotesize] {$5$};
\draw [very thick] (2,-2) circle (2pt);
\filldraw [black] (2,-2) circle (2pt);
\draw (2,-2.2) node[below,font=\footnotesize] {$6$};
\draw [very thick] (3,-2) circle (2pt);
\filldraw [black] (3,-2) circle (2pt);
\draw (3,-2.2) node[below,font=\footnotesize] {$7$};
\draw [very thick] (4,-2) circle (2pt);
\filldraw [black] (4,-2) circle (2pt);
\draw (4,-2.2) node[below,font=\footnotesize] {$8$};
\draw [very thick] (5,-2) circle (2pt);
\filldraw [black] (5,-2) circle (2pt);
\draw (5,-2.2) node[below,font=\footnotesize] {$9$};
\end{tikzpicture}

}

\vspace{.5cm}
 \subfigure[][Friendship graph $F_9$]{

\begin{tikzpicture}[scale=0.7]

\draw (0,0) -- (2,1)--(2,-1)--(0,0);
\draw (0,0) -- (-2,1)--(-2,-1)--(0,0);
\draw (0,0) -- (-1,2)--(1,2)--(0,0);
\draw (0,0) -- (-1,-2)--(1,-2)--(0,0);


\filldraw [black] (0,0) circle (2.5pt);
\filldraw [black] (2,1) circle (2.5pt);
\filldraw [black] (2,-1) circle (2.5pt);
\filldraw [black] (-2,-1) circle (2.5pt);
\filldraw [black] (-2,1) circle (2.5pt);
\filldraw [black] (-1,2) circle (2.5pt);
\filldraw [black] (-1,-2) circle (2.5pt);
\filldraw [black] (1,2) circle (2.5pt);
\filldraw [black] (1,-2) circle (2.5pt);

\end{tikzpicture}

}\qquad \qquad \subfigure[][Wheel graph $W_9$]{

\begin{tikzpicture}[scale=.7]
\GraphInit[vstyle=Simple]
    \SetGraphUnit{1}
    \tikzset{VertexStyle/.style = {shape = circle,fill = black,minimum size = 2.5pt,inner sep=1.5pt}}
  \begin{scope}[xshift=12cm]
\grWheel[RA=2.2]{9}%
\end{scope}
\end{tikzpicture}

}\qquad  \qquad \subfigure[][Complete graph $K_9$]{\label{subfig:complete}
\begin{tikzpicture}[scale=0.35]
\GraphInit[vstyle=Simple]
    \SetGraphUnit{1}
    \tikzset{VertexStyle/.style = {shape = circle,fill = black,minimum size = 2pt,inner sep=1.5pt}}
  \begin{scope}[xshift=12cm]
  \grComplete[RA=4.5]{9}
  \end{scope}
\end{tikzpicture}
}

\end{center}
\caption{Illustration of graphs with $\Der(\A(G))=\{0\}$.}\label{FIG:graphsderzero}
\end{figure}

\bigskip
\begin{exa} {\bf Path graph $P_n$}. Consider the path graph with $n$ vertices $P_n$, see Fig. \ref{FIG:graphsderzero}(a), and denote its adjacency matrix by $A_{P_n}$. Since $P_n$ is an example of a twin-free graph, if $d\in\Der(\A(P_{n}))$ then by Theorem \ref{prop:necessary} we have $d=0$. For the sake of comparison, let us show the same result by means of the criteria obtained by \cite{COT}. It is known that $\det(A_{P_{n}})\neq 0$ if, and only, if $n$ is even, see \cite[Proposition 2.4]{Abdollahi}. So \cite[Theorem 2.1]{COT} implies that $d$ is zero whether $n$ is even. Let us now assume $n$ odd, where a straightforward calculation gives $\rank(A_{P_n})=n-1$, and then the results in \cite{COT} apply. In this case, it is not difficult to see that we can write:
$$e_n^2=\sum_{i=1}^{n-1} b_i\, e_i^2,$$
where
$$b_{i}=\left\{
\begin{array}{cl}
0, &\text{ if }i=2k, \text{ for }k\in\{1,\ldots,(n-1)/2\},\\[.2cm]
(-1)^{k-1},&\text{ if }i=n-2k, \text{ for }k\in\{1,\ldots,(n-1)/2\}.  
\end{array}\right.
$$

\smallskip
\noindent
Therefore we can use Lemma 2.2 in \cite{COT} to conclude that $d$ is zero for $n$ odd too. 
\end{exa}

\smallskip
\begin{exa} {\bf Wheel graph $W_n$}. Consider the wheel graph $W_n$, which is a graph with $n$ vertices, $n\geq 4$, formed by connecting a single vertex, called center, to all the vertices of an $(n-1)$-cycle, see Fig. \ref{FIG:graphsderzero}(c). We denote the central vertex by $n$. If $A_{W_n}$ is the adjacency matrix of $W_n$ it is known that $\det(A_{W_n})= 0$ if, and only, if $n \equiv 1\,\,  (\bmod {\,4})$, see \cite[Corollary 2.2]{jeib} for more details. Thus we consider $W_n$ such that $n \equiv 1\,\, (\bmod{\,4})$, since otherwise the space of derivations is zero by \cite[Theorem 2.1]{COT}. We point out that whether $n \equiv 1\,\,(\bmod{\,4})$ it is possible to show that $\rank(A_{W_n})=n-2$, so the results in  \cite{COT} do not apply. We emphasize that Theorem \ref{prop:necessary} is enough to guarantee, also in this case, that $\Der(\A(W_n))=\{0\}$.
\end{exa}

At this point one may ask: what we can say about $\Der(\A(G))$ for those graphs $G$ such that $\Gamma_3(G)\neq \emptyset$? The answer of this question is the spirit behind the following theorem, which provides a natural and intrinsic characterization of $\Der(\A(G))$ depending on the twin classes with at least three elements of $G$.


\begin{theorem}\label{theo:principal}
Let $G=(V,E)$ be a finite graph, with $|V|\geq 3$, and let $d\in \Der(\A(G))$. If  $\Gamma_3(G)=\{\T_1,\ldots,\T_m\}$, then $d$ is in one of the following forms up to basis permutation:

\smallskip
\begin{equation}\label{eq:matdermpartite}
 \left( 
 \begin{tikzpicture}[scale=0.7,baseline=-1mm] 
\draw (1,4) node[font=\large] {$A_{1}(d)$};
\draw (3,2) node[font=\large] {$A_{2}(d)$};
\draw (5,0) node[font=\large] {$\ddots$};
\draw (7,-2) node[font=\large] {$A_{m}(d)$};
\draw (9,-4) node[font=\large] {$\bf 0$};

\draw (3,4) node[font=\large] {$\bf 0$};
\draw (5,4) node[font=\large] {$\cdots$};
\draw (7,4) node[font=\large] {$\bf 0$};
\draw (9,4) node[font=\large] {$\bf 0$};

\draw (1,2) node[font=\large] {$\bf 0$};
\draw (5,2) node[font=\large] {$\cdots$};
\draw (7,2) node[font=\large] {$\bf 0$};
\draw (9,2) node[font=\large] {$\bf 0$};

\draw (1,0) node[font=\large] {$\vdots$};
\draw (3,0) node[font=\large] {$\vdots$};
\draw (7,0) node[font=\large] {$\vdots$};
\draw (9,0) node[font=\large] {$\vdots$};

\draw (1,-2) node[font=\large] {$\bf 0$};
\draw (3,-2) node[font=\large] {$\bf 0$};
\draw (5,-2) node[font=\large] {$\cdots$};
\draw (9,-2) node[font=\large] {$\bf 0$};

\draw (1,-4) node[font=\large] {$\bf 0$};
\draw (3,-4) node[font=\large] {$\bf 0$};
\draw (5,-4) node[font=\large] {$\cdots$};
\draw (7,-4) node[font=\large] {$\bf 0$};

\draw [dashed, line width=.05mm] (2,5) -- (2,-5); 
\draw [dashed, line width=.05mm] (4,5) -- (4,-5); 
\draw [dashed, line width=.05mm] (6,5) -- (6,-5); 
\draw [dashed, line width=.05mm] (8,5) -- (8,-5); 

\draw [dashed, line width=.05mm] (0,3) -- (10,3); 
\draw [dashed, line width=.05mm] (0,1) -- (10,1); 
\draw [dashed, line width=.05mm] (0,-1) -- (10,-1); 
\draw [dashed, line width=.05mm] (0,-3) -- (10,-3); 

\end{tikzpicture} \right),
\end{equation}

\bigskip
\noindent
where ${\bf 0}$ denotes blocks of zeros, $A_{\ell}(d):= U_{\ell}(d) - U_{\ell}(d)^T$, and $U_{\ell}(d):=(u_{ij})$ is the $a_{\ell}$-dimensional upper triangular matrix, with $a_{\ell}:=|\T_{\ell}|$ for $\ell\in \{1,\ldots,m\}$, given by
\begin{equation}\label{eq:matrizes}
u_{ij}:=\left\{
\begin{array}{cl}
0,&\text{ if }i\geq j,\\[.3cm]
 d_{(s(\ell)+i)\, (s(\ell)+ j)}, & \text{ if }i<j;
\end{array}\right.
\end{equation}
\noindent
where $s(i):=\sum_{j=1}^{i -1} a_j$,   for  $i \in \{2, \ldots, m\}$ and $s(1):=0$. Moreover, $d$ satisfies the additional condition: for $\ell\in\{1,\ldots,m\}$,
\begin{equation}\label{eq:enunciadocond}
\sum_{k\in \mathcal{T}_{\ell}}d_{ki}=0, \text{ for any }i\in \mathcal{T}_{\ell}.
\end{equation}
\end{theorem}


\begin{exa} {\bf Complete $m$-partite graphs.} Let $K_{a_1, a_2, \ldots, a_m}$ be a complete $m$-partite graph with $n$ vertices, $m\geq 2$, and a partition of vertices $\{V_1,\ldots, V_m\}$ with respective sizes $a_1,  \ldots, a_m$, where $a_i \geq 3$ for $i\in \{1,2,\ldots,m\}$. See Fig. \ref{fig:npartite} for an illustration of this type of graph. Since $\sum_{i=1}^{m}a_i =n$, we have that $\A(K_{a_1, a_2, \ldots, a_m})$ is an $n$-dimensional evolution algebra with rank equals to $m$. We can see this from the adjacency matrix of the graph, $A_{K_{a_1, a_2, \ldots, a_m}}$, which is given by

\begin{equation*}
 \left( 
 \begin{tikzpicture}[scale=0.7,baseline=7mm] 
 
\draw (1,4) node[font=\large] {${\bf 0}_{a_1\times a_1}$};
\draw (3,2) node[font=\large] {$ {\bf 0}_{a_2\times a_2} $};
\draw (5,0) node[font=\large] {$\ddots$};
\draw (7,-2) node[font=\large] {$ {\bf 0}_{a_m\times a_m}$};

\draw (3,4) node[font=\large] {${\bf 1}_{a_1\times a_2}$};
\draw (5,4) node[font=\large] {$\cdots$};
\draw (7,4) node[font=\large] {${\bf 1}_{a_1\times a_m} $};

\draw (1,2) node[font=\large] {${\bf 1}_{a_2\times a_1}$};
\draw (5,2) node[font=\large] {$\cdots$};
\draw (7,2) node[font=\large] {$ {\bf 1}_{a_2\times a_m} $};

\draw (1,0) node[font=\large] {$\vdots$};
\draw (3,0) node[font=\large] {$\vdots$};
\draw (7,0) node[font=\large] {$\vdots$};

\draw (1,-2) node[font=\large] {${\bf 1}_{a_m\times a_1}$};
\draw (3,-2) node[font=\large] {$ {\bf 1}_{a_m\times a_2}$};
\draw (5,-2) node[font=\large] {$\cdots$};


\draw [dashed, line width=.05mm] (2,5) -- (2,-3); 
\draw [dashed, line width=.05mm] (4,5) -- (4,-3); 
\draw [dashed, line width=.05mm] (6,5) -- (6,-3); 

\draw [dashed, line width=.05mm] (0,3) -- (8,3); 
\draw [dashed, line width=.05mm] (0,1) -- (8,1); 
\draw [dashed, line width=.05mm] (0,-1) -- (8,-1); 

\end{tikzpicture} \right)
\end{equation*}

\vspace{.5cm}
\noindent
where ${\bf 0}_{a\times b}$ (${\bf 1}_{a\times b}$) denotes the matrix of size $a\times b$ and elements equal to $0$ (equal to $1$). The important point to note here is that although the criteria developed in \cite{COT} can not be applied whether $m\leq n-2$, Theorem \ref{theo:principal} works. To see this, it is enough to notice that $\Gamma_3(K_{a_1, a_2, \ldots, a_m})=\{V_1,\ldots, V_m\}$, i.e., each subset of the partition of the graph is a twin class as well. Therefore, if $d\in \Der\left(\A(K_{a_1, a_2, \ldots, a_m})\right)$, then  $d$ is in one of the forms given by \eqref{eq:matdermpartite} up to basis permutation. For the sake of illustration, let us take $m=2$, and $a_1=a_2=3$. Theorem \ref{theo:principal} implies that if $d\in \Der\left(\A(K_{3, 3})\right)$, then $d$ is in one of the following forms up to basis permutation:

\smallskip
\begin{equation}\label{eq:matderbipartite}
 \left( 
 \begin{tikzpicture}[scale=0.35,baseline=-4mm] 

\draw (1,4) node[font=\small] {$0$};
\draw (3,4) node[font=\small] {$\alpha$};
\draw (5,4) node[font=\small] {$-\alpha$};
\draw (7,4) node[font=\small] {$0$};
\draw (9,4) node[font=\small] {$0$};
\draw (11,4) node[font=\small] {$0$};

\draw (1,2) node[font=\small] {$-\alpha$};
\draw (3,2) node[font=\small] {$0$};
\draw (5,2) node[font=\small] {$\alpha$};
\draw (7,2) node[font=\small] {$0$};
\draw (9,2) node[font=\small] {$0$};
\draw (11,2) node[font=\small] {$0$};

\draw (1,0) node[font=\small] {$\alpha$};
\draw (3,0) node[font=\small] {$-\alpha$};
\draw (5,0) node[font=\small] {$0$};
\draw (7,0) node[font=\small] {$0$};
\draw (9,0) node[font=\small] {$0$};
\draw (11,0) node[font=\small] {$0$};

\draw (1,-2) node[font=\small] {$0$};
\draw (3,-2) node[font=\small] {$0$};
\draw (5,-2) node[font=\small] {$0$};
\draw (7,-2) node[font=\small] {$0$};
\draw (9,-2) node[font=\small] {$\beta$};
\draw (11,-2) node[font=\small] {$-\beta$};

\draw (1,-4) node[font=\small] {$0$};
\draw (3,-4) node[font=\small] {$0$};
\draw (5,-4) node[font=\small] {$0$};
\draw (7,-4) node[font=\small] {$-\beta$};
\draw (9,-4) node[font=\small] {$0$};
\draw (11,-4) node[font=\small] {$\beta$};

\draw (1,-6) node[font=\small] {$0$};
\draw (3,-6) node[font=\small] {$0$};
\draw (5,-6) node[font=\small] {$0$};
\draw (7,-6) node[font=\small] {$\beta$};
\draw (9,-6) node[font=\small] {$-\beta$};
\draw (11,-6) node[font=\small] {$0$};

\draw [dashed, line width=.05mm] (6,5) -- (6,-7); 
\draw [dashed, line width=.05mm] (0,-1) -- (12,-1); 

\end{tikzpicture} \right),
\end{equation}

\smallskip
for $\alpha,\beta \in \mathbb{R}$.
\end{exa}

\section{Proofs}

\subsection{Preliminary results}

Since we are dealing with graphs, let us provide another look for the conditions \eqref{eq:der1} and \eqref{eq:der2}.

\begin{prop}\label{prop:conditions}
Let $G=(V,E)$ be a finite graph, with $|V|\geq 3$, and let $d\in \Der\left(\A(G)\right)$. Then $d$ satisfies the following conditions:
\begin{enumerate}[label=(\roman*)]
\item If $i,j\in V$, $i\neq j$, and $\mathcal{N}(i) \cap \mathcal{N}(j) \neq \emptyset$ then $d_{ij}=-d_{ji}$.
\item If $i,j\in V$, $i\neq j$, and $\mathcal{N}(i) \cap \mathcal{N}(j)^c \neq \emptyset$ then $d_{ji}=0$.
\item If $i,j\in V$, $i\neq j$, and $\mathcal{N}(i) \cap \mathcal{N}(j) = \emptyset$ then $d_{ij}=d_{ji}=0$.
\item For any $i\in V$, we have
$$
\sum_{k\in \mathcal{N}(i)} d_{kj}=\left\{
\begin{array}{cl}
0,&\text{ if }j\notin \mathcal{N}(i),\\[.2cm]
2d_{ii},&\text{ if }j\in \mathcal{N}(i).
\end{array}\right.
$$ 
\end{enumerate}
\end{prop}

\begin{proof}
In order to prove (i) consider $i,j \in V$, with $i\neq j$, such that $\mathcal{N}(i)\cap \mathcal{N}(j) \neq \emptyset$, and let $k\in \mathcal{N}(i)\cap \mathcal{N}(j)$. Since $k$ is a neighbor of both $i$ and $j$ we have $a_{ik}=a_{jk}=1$. Therefore condition \eqref{eq:der1} implies $d_{ij} + d_{ji}=0$, or $d_{ij}=-d_{ji}$. Analogously, assume now that for some $i,j\in V$ we have $\mathcal{N}(i)\cap \mathcal{N}(j)^c \neq \emptyset$, and let $m\in \mathcal{N}(i)\cap \mathcal{N}(j)^c$. This means that $a_{im}=1$ while $a_{jm}=0$. This implies, by \eqref{eq:der1}, that $d_{ji}=0$ and the proof of (ii) is completed. Item (iii) is a consequence of (ii) since $\mathcal{N}(i)\cap \mathcal{N}(j)=\emptyset$ implies $\mathcal{N}(i)\cap \mathcal{N}(j)^c\neq \emptyset$, and $\mathcal{N}(j)\cap \mathcal{N}(i)^c\neq \emptyset$. Finally, item (iv) may be obtained by observing in \eqref{eq:der2} that for any $i\in V$, $a_{ik}=1$ if, and only if $k\in \mathcal{N}(i)$.

\end{proof}

\bigskip
In other words, condition (i) refers to pairs of vertices having at least one neighbor in common; (ii) is related to pairs of vertices for which one of them has a particular vertex as a neighbor and the other does not; (iii) is the case for which the pair of vertices do not have any neighbor in common; and (iv) gives an expression valid for any vertex of the graph. Thus Proposition \ref{prop:conditions} provides some conditions to be checked in order to look for the derivations of a given evolution algebra $\A(G)$. Moreover, these conditions depends on the topology of the graph $G$. The following result may be seen as a direct consequence of Proposition \ref{prop:conditions}.

\smallskip
\begin{corollary}\label{cor:conditions}
Let $G=(V,E)$ be a finite graph, with $|V|\geq 3$, and let $d\in \Der\left(\A(G)\right)$. Then, the derivation $d$ satisfies the following conditions:
\begin{enumerate}[label=(\roman*)]
\item For any $i\in V$,
\begin{equation}
d_{ii}=\frac{1}{2\, \deg(i)} \sum_{k\in\mathcal{N}(i)}d_{kk}.
\end{equation}
\item If $i\sim_{t} j$ then $d_{ii}=d_{jj}$.
\end{enumerate}
\end{corollary}

\begin{proof}
Notice that (ii) is a direct consequence of (i). In order to prove (i), we consider Proposition \ref{prop:conditions}(iv) for $i,j\in V$, with $j\in\mathcal{N}(i)$, i.e.
\begin{equation}\label{eq:cor1}
2d_{ii}=\sum_{k\in\mathcal{N}(i)}d_{kj}.  
\end{equation}
Since Eq. \eqref{eq:cor1} holds for any $j\in \mathcal{N}(i)$, we have
$$
\sum_{j\in\mathcal{N}(i)}2d_{ii}=\sum_{j\in\mathcal{N}(i)}\sum_{k\in\mathcal{N}(i)} d_{kj},
$$
which can be written as
\begin{equation}\label{eq:cor2}
2 \deg(i) d_{ii} = \sum_{k\in\mathcal{N}(i)} d_{kk} + \sum_{\substack{j,k \in \mathcal{N}(i)\\j\neq k}}d_{kj}.
\end{equation}
Finally observe that the second term of the right side in Eq. \eqref{eq:cor2} vanishes because it includes for any $k,j \in \mathcal{N}(i)$ both $d_{jk}$ and $d_{kj}$. But $\mathcal{N}(k)\cap \mathcal{N}(j)\neq \emptyset$ so $d_{jk}=-d_{kj}$ by Proposition \ref{prop:conditions}(i). This completes the proof. 

\end{proof}

\smallskip
\begin{lemma}\label{lema:1}
Let $G=(V,E)$ be a finite graph, with $|V|\geq 3$, and let $d \in \Der\left(\A(G)\right)$. If $d_{ij}=d_{ji}=0$ for any $i,j\in V$, $i\neq j$, then $d_{ii}=0$ for any $i\in V$. 
\end{lemma}

\begin{proof}
Let $i\in V$ such that $d_{ii}\neq 0$ and let $j \in \mathcal{N}(i)$. By applying Proposition \ref{prop:conditions}(iv) we have that 
\begin{equation}\label{eq:lem1}
2d_{ii} = \sum_{k\in\mathcal{N}(i)} d_{k j} = d_{jj},
\end{equation}
where the last equality is by the hypothesis. On the other hand,  since $i \in \mathcal{N}(j)$ we apply again Proposition \ref{prop:conditions}(iv) and we obtain
\begin{equation}\label{eq:lem2}
2d_{jj} = \sum_{k\in\mathcal{N}(j)} d_{ki} = d_{ii}.
\end{equation}
Therefore \eqref{eq:lem1} and \eqref{eq:lem2} imply $2d_{ii}=(1/2)d_{ii}$, and then $d_{ii}=0$, which is a contradiction.

\end{proof}

\smallskip
\begin{lemma}\label{lemma:algum}
Let $G=(V,E)$ be a finite graph, with $|V|\geq 3$, and let $d \in \Der\left(\A(G)\right)$. If $d_{ij}\neq 0$, for some $i,j\in V$ with $i\neq j$, then $i\sim_{t} j$. 
\end{lemma}

\begin{proof}
Let $i,j \in V$ such that  $d_{ij}\neq 0$. By Proposition \ref{prop:conditions}(iii) we have $\mathcal{N}(i)\cap \mathcal{N}(j)\neq \emptyset$, but then Proposition \ref{prop:conditions}(i) implies $d_{ji}=-d_{ij}\neq 0$. Hence, since $d_{ij}\neq 0$ and $d_{ji}\neq 0$ we conclude by Proposition \ref{prop:conditions}(ii) that $\mathcal{N}(i)\cap \mathcal{N}(j)^c = \emptyset$ and $\mathcal{N}(i)^c\cap \mathcal{N}(j)= \emptyset$, that is, $\mathcal{N}(i)= \mathcal{N}(j)$.

\end{proof}

\smallskip
\begin{lemma}\label{prop:necessarylemma}
Let $G$ be a finite graph, with $|V|\geq 3$, and let $d\in \Der(\A(G))$. If $d_{k\ell}\neq 0$ for some $k,\ell\in V$, with $k\neq \ell$, then there exists a twin class $\T_1 \subset V,$ with $|\T_1|\geq 3$, such that $k,\ell \in \T_1$.
\end{lemma}

\begin{proof}
Without loss of generality assume $d_{12}\neq 0$. By Lemma \ref{lemma:algum} we have that $1\sim_t 2$. From now on we shall analyze two cases:\\

\noindent
{\bf Case 1.}  Assume that there exists $i\in V\setminus\{1,2\}$ such that $d_{1i}\neq 0$. Then, again by Lemma \ref{lemma:algum}, $1\sim_t i$ and then $\T_1 \supset  \{1,2,i\}$.\\

\noindent
{\bf Case 2.} Assume that $d_{1i}=0$ for any $i\in V\setminus\{1,2\}$. Since $d_{21}\neq 0$ because of Proposition \ref{prop:conditions}(i), we have the following possible situations:\\

\noindent
$\bullet$ Suppose that there exists $j\in V\setminus\{1,2\}$ such that $d_{2j}\neq 0$. Then we can conclude, as before, $\mathcal{N}(2)=\mathcal{N}(j)$, and then $\T_1\supset \{1,2,j\}$.\\

\noindent
$\bullet$ Suppose $d_{2j}=0$ for any $j\in V\setminus\{1,2\}$, since $1\sim_t 2$ and $|V|\geq 3$ then there exists $j\in V\setminus\{1,2\}$ such that $j\in \mathcal{N}(1)$. Without loss of generality we let $j=3$. Now, observe that for any $\ell \in \mathcal{N}(3)$ we have by Proposition \ref{prop:conditions}(i), $d_{2\ell}=-d_{\ell 2}$, but $d_{2\ell}=0$ so $d_{\ell 2}=0$ as well. Thus

\begin{equation}\label{eq:propmain1}
\sum_{\substack{\ell \in \mathcal{N}(3)\\\ell\neq 1,2}} d_{\ell 2} = 0.
\end{equation}

In a similar way we obtain

\begin{equation}\label{eq:propmain1-}
\sum_{\substack{\ell \in \mathcal{N}(3)\\\ell\neq 1,2}} d_{\ell 1} = 0.
\end{equation}

Proposition \ref{prop:conditions}(iv), applied for $i=3$, $j=1$, and $i=3$, $j=2$, implies respectively the following equations:

\begin{equation}\label{eq:propmain2}
2d_{33}=d_{11} + d_{21} + \sum_{\substack{\ell \in \mathcal{N}(3)\\\ell\neq 1,2}} d_{\ell1 },
\end{equation}

and

\begin{equation}\label{eq:propmain3}
2d_{33}= d_{12} + d_{22} + \sum_{\substack{\ell \in \mathcal{N}(3)\\\ell\neq 1,2}} d_{\ell 2} .
\end{equation}

By \eqref{eq:propmain2} and \eqref{eq:propmain3}, adding the expressions, we have $4d_{33} = 2d_{11}$. The last conclusion is a consequence of \eqref{eq:propmain1}, \eqref{eq:propmain1-}, $d_{12}+d_{21}=0$ (by Proposition \ref{prop:conditions}(i)), and $d_{11}=d_{22}$ (by Corollary \ref{cor:conditions}(ii)). On the other hand, Corollary \ref{cor:conditions}(i) implies 
$$2 \deg(1) d_{11}= \sum_{k\in \mathcal{N}(1)} d_{kk},$$
but, since a similar expression like \eqref{eq:propmain2} holds for any vertex in $\mathcal{N}(1)$, we have $2d_{kk}=d_{11} + d_{21}$, and then
$$2 \deg(1) d_{11}= \sum_{k\in \mathcal{N}(1)} \left(\frac{d_{11} + d_{21}}{2}\right)= \deg(1)  \left(\frac{d_{11} + d_{21}}{2}\right).$$
This in turns implies $3d_{11}=d_{21}$, and again from \eqref{eq:propmain2} we have $2d_{33}=4d_{11}$. Therefore, it should be $d_{11}=0$ which implies $d_{21}=0$. But this is a contradiction.

\end{proof}

\smallskip
\begin{lemma}\label{lema:final}
Let $G=(V,E)$ be a finite graph, with $|V|\geq 3$, let $\T_1\subset V$ be a twin class of $G$, and let $d\in \Der(\A(G))$. Then 
\begin{enumerate}[label=(\roman*)]
\item $d_{ik}=d_{ki}=0$ for any $i\in \T_1$ and $k\in \T_1^c$;
\item $d_{ii}=0$, for any $i\in \T_1$;
\item $d_{\ell\ell}=0$, for any $\ell \in \mathcal{N}(\T_1)$;
\item $d_{ij}=-d_{ji}$, for any $i,j\in \T_1$, with $i\neq j$.
\end{enumerate}

\end{lemma}

\begin{proof}
First we shall prove (i). Consider $i\in \T_1$ and $k\in \T_1^c$. 
\begin{itemize}
\item If $\mathcal{N}(i)\cap \mathcal{N}(k) = \emptyset$ then Proposition \ref{prop:conditions}(iii) implies $d_{ik}=d_{ki}=0$.
\item If $\mathcal{N}(i)\cap \mathcal{N}(k) \neq \emptyset$ we have by Proposition \ref{prop:conditions}(i) that $d_{ik}=-d_{ki}$. Then, since $\T_1$ is a twin class, we have two options:
\begin{itemize}
\item or $\mathcal{N}(i)\cap \mathcal{N}(k)^c \neq \emptyset$, and then Proposition \ref{prop:conditions}(ii) implies $d_{ki}=0$;
\item or $\mathcal{N}(k)\cap \mathcal{N}(i)^c \neq \emptyset$, and then Proposition \ref{prop:conditions}(ii) implies $d_{ik}=0$.
\end{itemize}
In both cases, one can conclude $d_{ik}=d_{ki}=0$ and the proof of (i) is completed.
\end{itemize}  

\smallskip
Now let us prove (ii) and (iii) together. Let $\T_1=\{1,\ldots,m\}$, let $i\in \T_1$ and let $\ell \in \mathcal{N}(i)$. By Proposition \ref{prop:conditions}(iv) we have
$$2d_{\ell \ell} =\sum_{k\in \mathcal{N}(\ell)} d_{ki},$$
but
$$\sum_{k\in \mathcal{N}(\ell)} d_{ki} = \sum_{k\in \mathcal{N}(\ell)\cap \T_1} d_{ki} + \sum_{k\in \mathcal{N}(\ell)\cap \T_1^c} d_{ki}=\sum_{k=1}^{m} d_{ki},$$
where the last equality is due to $\mathcal{N}(\ell)\cap \T_1 = \T_1$, together to the fact that $\sum_{k\in \mathcal{N}(\ell)\cap \T_1^c} d_{ki}$ vanishes by (i). So, in general we have for any $i\in \T_1$, and $\ell \in \mathcal{N}(i)$:
\begin{equation}\label{eq:some}
2 d_{\ell \ell} =\sum_{k=1}^{m} d_{ki},
\end{equation}
and taking the sum over all $i\in \T_1$ in Eq. \eqref{eq:some} one get
\begin{equation}\label{eq:some2}
2 m d_{\ell \ell} = \sum_{k=1}^m d_{kk} = m d_{11}.
\end{equation}
The first equality of Eq. \eqref{eq:some2} is due to Proposition \ref{prop:conditions}(i), while the second one is due to Corollary \ref{cor:conditions}(ii). On the other hand, Corollary \ref{cor:conditions}(i) and \eqref{eq:some2} imply for $i\in \T_1$ 
\begin{equation}\label{eq:some3}
2 \deg(i) d_{ii} = \sum_{k\in \mathcal{N}(i)} d_{kk} = \deg(i) d_{\ell \ell}.
\end{equation}
Therefore we have $2 d_{\ell \ell} = d_{ii}$ (see Eq. \eqref{eq:some2}) and $d_{ii}=(1/2)d_{\ell \ell}$ (see Eq. \eqref{eq:some3}) so $d_{ii}=d_{\ell \ell}=0$, for $i\in \T_1$ and $\ell \in \mathcal{N}(i)$. This completes the proof of (ii) and (iii).

\smallskip
Finally, we notice that (iv) is a direct consequence of Proposition \ref{prop:conditions}(i).

\end{proof}

\bigskip
\subsection{Proofs of the theorems}

\subsubsection{Proof of Theorem \ref{prop:necessary}}

Assume $d\neq 0$. By Lemma \ref{lema:1} we have that $d_{ij}\neq 0$ for some $i,j\in V$, $i\neq j$. Then Theorem \ref{prop:necessary} is a consequence of Lemma \ref{prop:necessarylemma}.

\smallskip
\subsubsection{Proof of Theorem \ref{theo:principal}}

Let  $\Gamma_3(G)=\{\T_1,\T_2,\ldots,\T_m\}$, where $m\in \{1,\ldots,\lfloor n/3 \rfloor\}$. Moreover, let $a_i :=|\T_i|$ and without loss of generality relabel the vertices of the graph as
$$\T_i=\{s(i)+1,s(i)+2,\ldots,s(i)+a_i\},$$
for $i\in \{1,\ldots,m\}$, where 
$$s(i):=\sum_{j=1}^{i-1}a_j,$$
whether $i>1$ and $s(1):=0$. In what follows we denote $\T := \bigcup_{i=1}^m \T_i$. Let $d\in \Der(\A(G))$.

The main idea behind the proof is to discover the matrix representation of $d$ by considering the twin classes of $\Gamma_3(G)$ one at a time. First let us consider $\T_1$. We shall see that we can find the values of $d_{ij}$ provided $i\in \T_1$ or $j\in \T_1$. In order to do it, we use Lemma \ref{lema:final}. Note that Lemma \ref{lema:final}(i) implies $d_{ij}=0$ if $i\in \T_1$ and $j\in \T_1^c$ or, if $j\in \T_1$ and $i\in \T_1^c$. Now we have to check the values of $d_{ij}$ for $i,j\in \T_1$. We observe that $d_{ii}=0$ for any $i\in \T_1$ thanks to Lemma \ref{lema:final}(ii). On the other hand, $d_{ij}=-d_{ji}$ for $i,j\in \T_1$ by Lemma \ref{lema:final}(iv). Hence, we obtain $A_1(d)=U_1(d)-U_{1}(d)^T$, where $U_1(d)$ is the $a_1$-dimensional matrix with elements $(u_{ij})$ satisfaying \eqref{eq:matrizes} and \eqref{eq:enunciadocond}. The latest condition is a consequence of Proposition \ref{prop:conditions}(iv). Summarizing, part of the representation matrix of $d$ has been obtained by observing the vertices in $\T_1$, see Fig. \ref{fig:exploration} for an illustration of the first step of this procedure. 

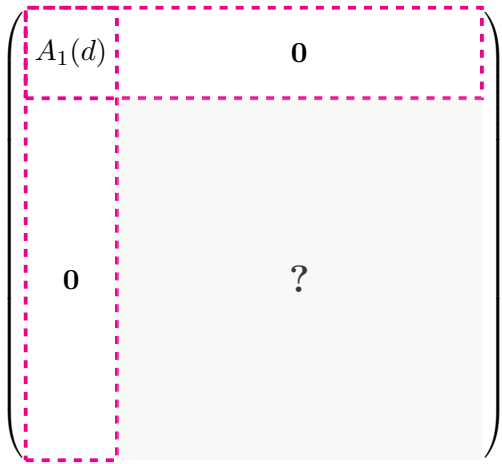
\begin{figure}[h!]
\[ \left( \begin{tikzpicture}[scale=0.6,baseline=-1mm] 
\draw [magenta, dashed, line width=.5mm] (0,3) rectangle (10,5); 
\draw [magenta, dashed, line width=.5mm] (2,5) rectangle (0,-5); 
\draw (1,4) node[font=\large] {$A_{1}(d)$};
\draw (6,4) node[font=\large] {$\bf 0$};
\draw (1,-1) node[font=\large] {$\bf 0$};
\draw (6,-1) node[font=\LARGE] {{\bf ?}};

\fill[gray!20,nearly transparent] (2,3) -- (10,3) -- (10,-5) -- (2,-5) -- cycle;

\end{tikzpicture} \right)
\]
\caption{First step in the exploration process of $d$.}\label{fig:exploration}
\end{figure}

It is not difficult to see that the same procedure can be applied to the second twin class $\T_2$ of $\Gamma_3(G)$. This in turns, allows us to discover the expression of a second part of the matrix of $d$, see Fig. \ref{fig:final}(left side). By continuing in this way, we can discover the part of $d$ associated to all the vertices in $\mathcal{T}$, see Fig. \ref{fig:final}(right side).

\begin{figure}[h!]
\[\left( \begin{tikzpicture}[scale=0.6,baseline=-1mm] 
\draw [magenta, dashed, line width=.5mm] (2,3) rectangle (10,1); 
\draw [magenta, dashed, line width=.5mm] (2,3) rectangle (4,-5); 

\fill[gray!20,nearly transparent] (4,1) -- (10,1) -- (10,-5) -- (4,-5) -- cycle;

\draw (1,4) node[font=\large] {$A_{1}(d)$};
\draw (7,4) node[font=\large] {$\bf 0$};
\draw (1,-2) node[font=\large] {$\bf 0$};
\draw (3,4) node[font=\large] {$\bf 0$};
\draw (1,2) node[font=\large] {$\bf 0$};
\draw (3,2) node[font=\large] {$A_{2}(d)$};
\draw (7,2) node[font=\large] {$\bf 0$};
\draw (3,-2) node[font=\large] {$\bf 0$};
\draw (7,-2) node[font=\LARGE] {{\bf ?}};

\draw [dashed, line width=.05mm] (0,3) -- (2,3); 
\draw [dashed, line width=.05mm] (0,1) -- (2,1);
\draw [dashed, line width=.05mm] (2,3) -- (2,5);  
\draw [dashed, line width=.05mm] (4,3) -- (4,5);  

\end{tikzpicture} \right)\quad \qquad \left( \begin{tikzpicture}[scale=0.6,baseline=-1mm] 
\draw [magenta, dashed, line width=.5mm] (8,-5) rectangle (10,-3); 

\fill[gray!20,nearly transparent] (8,-3) -- (10,-3) -- (10,-5) -- (8,-5) -- cycle;

\draw (1,4) node[font=\large] {$A_{1}(d)$};
\draw (3,2) node[font=\large] {$A_{2}(d)$};
\draw (5,0) node[font=\large] {$\ddots$};
\draw (7,-2) node[font=\large] {$A_{m}(d)$};
\draw (9,-4) node[font=\LARGE] {{\bf ?}};

\draw (3,4) node[font=\large] {$\bf 0$};
\draw (5,4) node[font=\large] {$\cdots$};
\draw (7,4) node[font=\large] {$\bf 0$};
\draw (9,4) node[font=\large] {$\bf 0$};

\draw (1,2) node[font=\large] {$\bf 0$};
\draw (5,2) node[font=\large] {$\cdots$};
\draw (7,2) node[font=\large] {$\bf 0$};
\draw (9,2) node[font=\large] {$\bf 0$};

\draw (1,0) node[font=\large] {$\vdots$};
\draw (3,0) node[font=\large] {$\vdots$};
\draw (7,0) node[font=\large] {$\vdots$};
\draw (9,0) node[font=\large] {$\vdots$};

\draw (1,-2) node[font=\large] {$\bf 0$};
\draw (3,-2) node[font=\large] {$\bf 0$};
\draw (5,-2) node[font=\large] {$\cdots$};
\draw (9,-2) node[font=\large] {$\bf 0$};

\draw (1,-4) node[font=\large] {$\bf 0$};
\draw (3,-4) node[font=\large] {$\bf 0$};
\draw (5,-4) node[font=\large] {$\cdots$};
\draw (7,-4) node[font=\large] {$\bf 0$};

\draw [dashed, line width=.05mm] (2,5) -- (2,-5); 
\draw [dashed, line width=.05mm] (4,5) -- (4,-5); 
\draw [dashed, line width=.05mm] (6,5) -- (6,-5); 
\draw [dashed, line width=.05mm] (8,5) -- (8,-3); 

\draw [dashed, line width=.05mm] (0,3) -- (10,3); 
\draw [dashed, line width=.05mm] (0,1) -- (10,1); 
\draw [dashed, line width=.05mm] (0,-1) -- (10,-1); 
\draw [dashed, line width=.05mm] (0,-3) -- (8,-3); 

\end{tikzpicture} \right)
\]
\caption{General idea to discover the matrix representation of $d$. The second step consist into discover those values of $d_{ij}$ such that $i\in \T_2$ or $j\in\T_2$ (left side). Once $d_{ij}$ is found for any $i,j$ such that $i\in \T$ or $j\in \T$ the remaining matrix to be discovered is formed by those $d_{ij}$ such that $i,j\in \T^c$ (right side).}\label{fig:final}
\end{figure}

The final step is to prove that the remaining block of the matrix representation of $d$, which is formed by those $d_{ij}$ such that $i,j \in \T^c$, is formed by zeros. By assuming $|\T|\leq n-3$ (the other cases will be considered later) we have that the remaining block is a square matrix with size of at least $3$. Notice that the existence of a pair of vertices $i,j\in \T^c$, with $i\neq j$, such that $d_{ij}\neq 0$, implies by Lemma \ref{prop:necessarylemma} the existence of another twin class with at least three vertices, but this is a contradiction by our construction. Therefore we have $d_{ij}=0$ for any $i,j\in \T^c$, with $i\neq j$. Now, let us discover the values of $d_{ii}$ whether $i\in \T^c$. Since we are considering a connected graph, we known that some vertices in $\T^c$ belongs to $\mathcal{N}\left( \T\right)$. So by Lemma \ref{lema:final}(iii) it holds $d_{ii} =0$ provided $i\in \mathcal{N}\left(\T\right)$. We shall see that the same result can be extended to any vertex $\ell$ such that $\ell \in \mathcal{N}(i)\cap \T^c$, for some $i\in \mathcal{N}(\T)$. Indeed, for the considered vertices we have $d_{ii}=0$ because of the previous argument, and we have
$$2d_{\ell \ell}=\sum_{k\in \mathcal{N}(\ell)}d_{k i}=d_{ii},$$
where the first equality comes from Proposition \ref{prop:conditions}(iv) (note that $\ell$ and $i$ are neighbors), and the second one is due to be $d_{ki}=0$ whether $k\neq i$. This means that $d_{\ell \ell}=0$. This argument can be extended to any vertex of $\T^c$. This is because since we are dealing with a connected graph, for any vertex of $\T^c$ there existe a path of vertices connecting it with a vertex in $\T$. Thus $d_{ii}=0$ for $i\in \T^c$ and the remaining block in our procedure is formed by zeros.

In order to finish the proof we have to prove that the remaining block, see Fig. \ref{fig:final}(right side), is formed by zeros also in the case of $|\T|\in\{n-1,n-2\}$. If $|\T|=n-1$, then $|\T^c|=1$ and the only vertex of $\T^c$,  labeled by $n$ according to our labels for $\T$, must be a neighbor of some vertex in $\T$. This in turns implies $d_{nn}=0$ by Lemma \ref{lema:final}(iii) and the proof is complete. On the other hand, assume $|\T|=n-2$, which implies $\T^c=\{n-1,n\}$. The same arguments than before allow us to conclude $d_{(n-1) (n-1)}=d_{nn}=0$. Furthermore, a suitable application of Proposition \ref{prop:conditions}(iv) lead us to $d_{(n-1) n}=d_{n (n-1)}=0$, so the proof is complete.

\section*{Acknowledgements}

This work was supported by FAPESP (Grants 2016/11648-0, 2017/19433-5, 2017/10555-0), CNPq (Grant 304676/2016-0), and Universidad de Antioquia. Part of this work was carried out during a visit of MLR at the ICMC - Universidade de S\~ao Paulo, and a visit of PC and PMR at the IM - Universidad de Antioquia. The authors wishes to thank these institutions for the hospitality and support. Special thanks are given to the referee, whose careful reading of the manuscript and valuable comments contributed to improve this paper.

\bigskip

\end{document}